\documentclass[12pt]{amsart}

\usepackage[utf8]{inputenc}

\usepackage{eucal}
\usepackage{upgreek}
\usepackage{pdfsync}
\usepackage[all, cmtip]{xy}
\usepackage{amsfonts}
\usepackage{amsmath}
\usepackage{amssymb}
\usepackage{amscd}
\usepackage{color,xcolor}
\usepackage{amsthm}
\usepackage{amsxtra}
\usepackage{bm}
\usepackage{bbm}
\usepackage{graphicx}
\usepackage[hmargin=3cm,vmargin=3cm, voffset=-25pt, footskip = 50pt]{geometry}
\usepackage{mathrsfs}
\usepackage[numbers,sort]{natbib}
\usepackage{isomath}
\usepackage{enumitem}
\usepackage{latexsym} 
\usepackage{dsfont}
\usepackage{url}
\usepackage{mathtools}
\usepackage{soul}

\newtheorem{thm}{Theorem}

\newtheorem{prop}{Proposition}
\newtheorem{lma}[prop]{Lemma}

\newtheorem{conj}[prop]{Conjecture}

\theoremstyle{definition}

\theoremstyle{remark}

\newtheorem{rmk}[prop]{Remark} 
\newtheorem{exa}[prop]{Example}

\def\mrm#1{{\mathrm{#1}}}
\def\bb#1{{\mathbb{#1}}}
\def\cl#1{{\mathcal{#1}}}

\newcommand{\R}{{\mathbb{R}}}
\newcommand{\Z}{{\mathbb{Z}}}
\newcommand{\C}{{\mathbb{C}}}
\newcommand{\Q}{{\mathbb{Q}}}
\newcommand{\N}{{\mathbb{N}}}

\newcommand{\bK}{{\mathbb{K}}}

\newcommand{\sm}[1]{C^\infty(#1)}

\newcommand{\om}{\omega}

\newcommand{\al}{\alpha}

\newcommand{\eps}{\epsilon}

\newcommand{\ol}[1]{\overline{#1}}
\renewcommand{\hat}[1]{\widehat{#1}}

\renewcommand{\geq}{\geqslant}
\renewcommand{\leq}{\leqslant}

\DeclareMathOperator{\id}{\mathrm{id}}

\DeclareMathOperator{\Ham}{\mathrm{Ham}}

\DeclareMathOperator{\Symp}{\mathrm{Symp}}

\DeclareMathOperator{\gr}{\mathrm{gr}}

\def\H2{H^{(2)}}
\def\F{\mathbb F}

\usepackage[hyperindex]{hyperref}

\newcommand{\esemail}{egor.shelukhin@umontreal.ca}

\begin{document}



\title[Symplectic Hilbert-Smith conjecture]{A symplectic Hilbert-Smith conjecture}

\author[Egor Shelukhin]{Egor Shelukhin\smallskip\\
\MakeLowercase{with an appendix by} Leonid Polterovich}

\email{\esemail}
\address{Department of Mathematics and Statistics, University of Montreal, C.P. 6128 Succ.  Centre-Ville Montreal, QC H3C 3J7, Canada \vspace{0.5cm}}

\email{ polterov@tauex.tau.ac.il}
\address{School of Mathematical Sciences, Tel Aviv University, Tel Aviv 69978, Israel}

\begin{abstract}
We prove new cases of the Hilbert-Smith conjecture for actions by natural homeomorphisms in symplectic topology. Specifically, we prove that the group of $p$-adic integers $\Z_p$ does not admit non-trivial continuous actions by Hamiltonian homeomorphisms, the $C^0$ limits of Hamiltonian diffeomorphisms, on symplectically aspherical symplectic manifolds. For a class of symplectic manifolds, including all standard symplectic tori, we deduce that a locally compact group acting faithfully by homeomorphisms in the $C^0$ closure of time-one maps of symplectic isotopies must be a Lie group. Our methods of proof differ from prior approaches to the question and involve barcodes and power operations in Floer cohomology.  They also apply to other natural metrics in symplectic topology, notably Hofer's metric. An appendix by Leonid Polterovich uses this to deduce obstructions on Hamiltonian actions by semi-simple $p$-adic analytic groups. \end{abstract}	


\date{June 25, 2024}
\maketitle

\section{Introduction and main results}

\subsection{Main results}

Hilbert's fifth problem, characterizing Lie groups purely topologically among the locally compact topological groups, was solved by Gleason \cite{Gleason1, Gleason2}, Montgomery-Zippin \cite{MZ, MZ-book}, with further work of Yamabe \cite{Yamabe1, Yamabe2} (see also \cite{Lee, Tao-book} for surveys). However, a generalization of Hilbert's fifth problem, called the Hilbert-Smith conjecture, is still largely open. Let us now recall its formulation. An action of a group $G$ by homeomorphisms on a topological space $X$ is called effective if the corresponding map $\rho: G \to \mrm{Homeo}(X)$ is injective. Endow $\mrm{Homeo}(X)$ with the compact-open topology; it is induced by the $C^0$ metric \[ d_{C^0}(f,g) = \max_{x \in M} d(f(x),g(x))\] in the case where $X$ is a closed manifold with an auxiliary Riemannian metric. An effective continuous action $\rho: G \to \mrm{Homeo}(X)$ of $G$ on $X$ by homeomorphisms is called a faithful action. 

\begin{conj}[Hilbert-Smith]\label{conj: HS}
If a locally compact topological group $G$ acts faithfully on a connected $n$-manifold $X$ then $G$ is a Lie group.
\end{conj}

The methods of the solution of Hilbert's fifth problem imply that this conjecture is equivalent to the following.

\begin{conj}\label{conj: Zp}
The topological group of $p$-adic integers $\Z_p$ does not act faithfully on a connected $m$-manifold.
\end{conj}

For one more equivalent version of Conjectures \ref{conj: HS} and \ref{conj: Zp} see \cite[Conjecture 1.3]{Pardon}.

\begin{rmk}\label{rmk: equiv}
It is important for this paper that for a given connected $m$-manifold $X$ this equivalence persists if we consider only actions by homeomorphisms in a fixed subgroup $\cl G$ of $\mrm{Homeo}(X).$
\end{rmk}

The Hilbert-Smith conjecture was proven for $m=1, m=2$ by Montgomery-Zippin \cite{MZ-book} and for $m=3$ by Pardon \cite{Pardon}. For actions by homeomorphisms satisfying additional regularity properties it was proven for $C^2$ actions by Bochner-Montgomery \cite{BM}, for $C^{0,1}$ (ie Lipschitz) actions by Repov\u{s}-\u{S}\u{c}epin \cite{RS}, for $C^{0,\frac{m}{m+2}+\eps}$ actions by Maleshich \cite{Maleshich} and for quasi-conformal actions by Martin \cite{Martin}. We also refer to work of Mj in a metric-theoretic setting \cite{Mj}. Negative results in more general settings, showing the sharpness of the conjecture, were achieved by Raymond-Williams \cite{RW}, Walsh \cite{Walsh}, and Wilson \cite{Wilson}.

An interesting implication of the validity of Conjecture \ref{conj: HS} is a resolution of the following conjecture. Let $\cl G$ be a closed subgroup of $\mrm{Homeo}(X).$ Let $f \in \cl G$ be almost-periodic in the sense that the subgroup $\langle f \rangle = \{ f^k\;|\; k \in \Z \}$ that it generates has compact closure. 

\begin{conj}\label{conj: ap}
Let $f \in \cl G$ be almost periodic. Then there exists an iteration $f^q,$ $q \in \N,$ of $f$ that is included in the image of a homomorphism $\R \to \cl G.$
\end{conj}

It is the goal of this paper to prove new cases of the Hilbert-Smith conjecture for actions on closed symplectic manifolds by homeomorphisms of symplectic nature. Recall that a symplectic manifold $(M,\om)$ is a manifold $M$ of dimension $m=2n$ with a closed, non-degenerate two-form $\om.$ All symplectic manifolds will be assumed to be connected.
We call $(M,\om)$ symplectically aspherical if $\langle [\om], A \rangle = 0, \langle c_1, A \rangle = 0, $ for all $A \in \pi_2(M),$ where $c_1 \in H^2(M;\Z)$ is the first Chern class of the complex vector bundle $(TM, J)$ for any almost complex structure $J$ compatible with $\om$ (according to Gromov \cite{Gromov-phc}, the space of such complex structures is contractible). 

Let $\Ham(M,\om)$ be the group of Hamiltonian diffeomorphisms: namely the time-one maps $\phi=\phi^1_H$ of isotopies $(\phi_H^t)_{t \in [0,1]}$ given by integrating the time-dependent vector fields $X^t_H$ given by the Hamiltonian construction \[ i_{X^t_H} \om = - d H_t,\] where $H_t(x) = H(t,x)$ is a Hamiltonian function, $H \in \sm{S^1 \times M, \R}$ for $S^1 = \R/\Z.$ It is easy to check that $\Ham(M,\om)$ is a subgroup of the group $\Symp_0(M,\om)$ of time-one maps $\psi = \psi^1$ of symplectic isotopies $(\psi^t)_{t \in [0,1]},$ $\psi^0 = \id,$ $(\psi^t)^* \om = \om$ for all $t \in [0,1].$ 

Let $M$ be closed and endow it with an auxiliary Riemannian metric. Following \cite{BHS}, we denote by $\ol{\Ham}(M,\om)$ the group of Hamiltonian homeomorphisms: it is the $C^0$-closure of $\Ham(M,\om)$ inside the group $\mrm{Homeo}(M)$ of homeomorphisms of $M.$ Working in this $C^0$ context we prove the following new cases of the Hilbert-Smith conjecture. 

\begin{thm}\label{cor: SHS1}
Let $(M,\om)$ be a closed symplectically aspherical manifold. There is no non-trivial continuous action of the $p$-adic group $\Z_p$ by Hamiltonian homeomorphisms. 
\end{thm}

Note that in this situation $\ol{\Ham}(M,\om)$ does not have finite subgroups \cite{SW-St} (see also \cite{P-nt}, \cite{AS-nt} for the smooth case), so the main novelty of Theorem \ref{cor: SHS1} is the non-existence of faithful actions. By work of Gleason, Montgomery-Zippin, and Yamabe \cite{Gleason1, Gleason2, MZ, MZ-book, Yamabe1, Yamabe2} Theorem \ref{cor: SHS1} is equivalent to the following.

\begin{thm}\label{cor: SHS1-gr}
If a locally compact group acts faithfully by Hamiltonian homeomorphisms on a closed symplectically aspherical manifold, then it is a torsion-free Lie group.
\end{thm}

\begin{rmk}\label{rmk: Rk} Note that torsion-free Lie groups are precisely the contractible ones. For example, for all natural $k,$ $\R^k$ admits continuous actions on every symplectic manifold by autonomous flows with disjoint supports. In fact $\R^k$ are the only connected torsion-free Lie groups acting faithfully, with respect to the $C^0$ topology \cite{BM}, by $\Ham(M,\om)$ on a closed symplectic manifold (see \cite{Delzant}). We expect that the same is true for actions by $\ol{\Ham}(M,\om)$ but will not prove it here.\end{rmk} 

\begin{rmk}\label{rmk: other}
Analyzing the proof of Theorem \ref{cor: SHS1}, we note that the $C^0$ continuity of the action was used only to show its continuity with respect to the $\gamma_{\Z}$ norm described below and we used the non-degeneracy property of the $\gamma_{\Q}$ norm on $\ol{\Ham}(M,\om).$ Hence the same argument verbatim proves that an analogous statement holds for $\Ham(M,\om)$ but endowed with the $\gamma_{\Z}$ norm or, using using \eqref{eq: ga hof} below, with the Hofer norm \cite{Hofer-metric, Polterovich-book}. (A further technical generalization is discussed in Remark \ref{rmk: completion}.) We record the case of the Hofer metric as Theorem \ref{thm: Hofer}. This result is used in Appendix \ref{app: Leonid} to provide further obstructions on Hamiltonian group actions of $p$-adic groups. \end{rmk}

\begin{thm}\label{thm: Hofer}
Let $(M,\om)$ be a closed symplectically aspherical manifold. There is no non-trivial Hofer-continuous action of the $p$-adic group $\Z_p$ by Hamiltonian diffeomorphisms.  
\end{thm}

Using the absence of torsion in $\Ham(M,\om)$ \cite{P-nt} it is immediate from \cite[Section 2.7.2]{Tao-book} (see also \cite[Proofs of Theorem 3.1 and 3.2]{Lee}) that Theorem \ref{thm: Hofer} implies a Hofer-geometric analogue of Theorem \ref{cor: SHS1-gr}.

\begin{thm}\label{cor: Hofer gr}
If a locally compact group acts Hofer-faithfully by Hamiltonian diffeomorphisms on a closed symplectically aspherical manifold, then it is a torsion-free Lie group.
\end{thm}

However, it appears to be unknown whether a Hofer-faithful action of a Lie group $G$, even $G=\R,$ via $\Ham(M,\om)$ must in fact be smooth (see also \cite[Remark 1.14]{PSST}). If this were true, we would obtain as in Remark \ref{rmk: Rk}, that such groups must be extensions of $\R^k$ by discrete torsion-free groups. Nevertheless, we expect to show in future work that this conclusion is indeed true for a special class of symplectically aspherical symplectic manifolds in a different way.

\begin{rmk} Our approach to Theorems \ref{cor: SHS1} and \ref{thm: Hofer} is motivated by a question of Polterovich arising from \cite{P-nt} (see also \cite{GG-ai} for a discussion and \cite{CGG-ga, SW-St} for partial progress on this question). Polterovich's question consists in proving that for $\phi \in \mrm{Ham}(M,\om) \setminus \{\id\},$ for $(M,\om)$ closed symplectically aspherical, there exists $\eps_{\phi} > 0$ such that for every $k \in \N,$ \[ ||\phi^k|| \geq \eps_{\phi},\] where $||-||$ stands for the Hofer norm. It could also stand for the $C^0$-norm and for the $\gamma_{\bK}$-norm for $\bK$ a unital coefficient ring. 
\end{rmk}

\begin{rmk}
The continuity of the action is essential. Indeed, $\Z_p$ is a subgroup of $\Q_p$ which is abstractly isomorphic to $\R$ as an abelian group (and even as a $\Q$ vector space) by a choice of Hamel bases. Hence any effective $\R$-action yields an effective, and hence non-trivial, $\Z_p$-action. It is easy to see that effective $\R$-actions by Hamiltonian diffeomorphisms exist even locally in the manifold: with support in any Darboux neighborhood symplectomorphic to a ball in the standard symplectic $\R^{2n}.$ (See also Remark \ref{rmk: new} for a related construction.)
\end{rmk}

By a result of \cite{AS-flux}, for a class $\cl C$ of closed symplectic manifolds including the standard symplectic $2n$-torus $T^{2n},$ the Flux homomorphism $\mrm{Flux}: \Symp_0(M,\om) \to H^1(M;\R)/\Gamma,$ where $\Gamma$ is the symplectic Flux group (it is a discrete subgroup of $H^1(M;\R)$), is uniformly $C^0$-continuous and extends to $\ol{\Symp}_0(M,\om)$ with kernel $\ol{\Ham}(M,\om).$ More generally, the class $\cl C$ contains all closed symplectically aspherical manifolds $(M,\om)$ such that either the Euler characteristic $\chi(M)$ does not vanish or the evaluation map \[\mrm{ev}: \pi_1(\Symp(M,\om)) \to \mrm Z (\pi_1(M))\] to the center of the fundamental group of $M$ is surjective. 

\begin{thm}\label{cor: SHS2}
Let $M$ be a closed symplectically aspherical manifold in class $\cl C.$ Then if a locally compact group $G$ acts on $M$ faithfully via $\ol{\Symp}_0(M,\om),$ then $G$ is a Lie group. \end{thm}



\begin{proof}[Proof of Theorem \ref{cor: SHS2}]
Suppose that the action is given by the homomorphism $\rho: G \to \ol{\Symp}_0(M,\om).$ If $G$ is not a Lie group, then it contains an embedded copy $H$ of $\Z_p.$ Set $\rho'=\rho|_H.$ Then the homomorphism $\theta = \ol{\mrm{Flux}} \circ \rho': H \to H^1(M;\R)/\Gamma$ is continuous. Moreover $\ker \theta \subset H$ is a compact subgroup isomorphic to $(\rho')^{-1}(\ol{\Ham}(M,\om)).$ By Theorem \ref{cor: SHS1} $\ker \theta = \{1\},$ whence $\theta$ is an embedding. However, $H^1(M;\R)/\Gamma,$ being a Lie group, does not contain an embedded copy of $\Z_p.$ This finishes the proof.\end{proof}


\begin{rmk}\label{rmk: 2d}
Theorem \ref{cor: SHS2} gives a new proof of the two-dimensional case of the Hilbert-Smith conjecture for every closed surface $X$ of genus $\geq 1.$ Indeed, given a faithful $\Z_p$ action $\rho: \Z_p \to \mrm{Homeo}(X),$ we may first average a non-atomic fully supported Borel measure along the $\Z_p$-action (using the Haar measure on $\Z_p$) and obtain an invariant such measure $\mu.$ By the Oxtoby-Ulam theorem \cite[Theorem 2]{OxtobyUlam}, $\mu$ is the image by a homemorphism of a standard smooth measure on $X$ coming from an area form $\om.$ Hence we may assume that $\Z_p$ acts by area-preserving homeomorphisms $\mrm{Homeo}(X,\om).$ Now, as the identity component $\mrm{Homeo}_0(X,\om)$ is open in $\mrm{Homeo}(X,\om),$ restricting the action to the subgroup $p^k \Z_p$ for $k$ sufficiently large, we may further assume that $\rho$ takes values in $\mrm{Homeo}_0(X,\om).$ However, by \cite[Theorem 5.1]{OhMuller}, $\mrm{Homeo}_0(X,\om) = \ol{\Symp}_0(X,\om),$ whence Theorem \ref{cor: SHS2} applies. We note that we used a general reduction of the Hilbert-Smith conjecture to the measure-preserving case, which could be useful in further settings.
\end{rmk}

\begin{rmk}\label{rmk: new}
It is easy to see that our results about the Hilbert-Smith conjecture in dimension at least $4$ do not follow from prior results. Indeed, considering a small neighborhood $U$ in $M$ symplectomorphic to the stadard ball $B_{\eps} \subset \R^{2n}$ of radius $\eps,$ we recall the definition of the infinite twist $\phi_{\rho}$ associated to a function $\rho \in \sm{(0,\eps]},$ $\rho(r) = 0$ near $r=\eps$ and $\rho(r) \to \infty$ as $r \to 0.$ Let $\rho_{n}$ denote a smooth function constant near $r=0$ and coinciding with $\rho$ on $[\eps/n,\eps].$ Then $H_n(x) = \rho_n(|x|)$ is a smooth Hamiltonian with compact support in $B_{\eps}$ and $\phi^1_{H_n}$ converges in the $C^0$ metric to a Hamiltonian homeomorphism as $n \to \infty$. Embedding $B_{\eps}$ into a Darboux chart of a closed symplectic manifold $(M,\om)$, we obtain $\phi \in \ol{\Ham}(M,\om).$ Now for $\rho(r) \to \infty$ sufficiently fast as $r \to 0$, we can see that $\phi$ does not satisfy any H\"older or quasi-conformal regularity properties.   
\end{rmk}

The proof of Theorem \ref{cor: SHS1} yields a slightly more general statement. (See description of $\gamma_{\F_q}$ below.)

\begin{prop}\label{prop: conv to zero}
Let $(M,\om)$ be a closed symplectically aspherical manifold. If a Hamiltonian homeomorphism $\phi \in \ol{\Ham}(M,\om)$ admits $q$-th roots $\psi_q \in \ol{\Ham}(M,\om),$ $(\psi_q)^q = \phi,$ for $q$ in an infinite increasing sequence $\cl Q$ of primes, then $\gamma_{\F_q}(\psi_q) \to 0$ as $q \to \infty.$
\end{prop}

Theorem \ref{cor: SHS1} follows by observing that if $\phi$ topologically generates a non-trivial $\Z_p$-action, then the conclusion of Proposition \ref{prop: conv to zero} is violated. Indeed, a general idea of this approach is that {\em small subgroups in $\Z_p$ imply large roots of homeomorphisms}. Note that the only known examples of $\phi \in {\Ham}(M,\om)$ satisfying the condition of Proposition \ref{prop: conv to zero} are autonomous Hamiltonian diffeomorphisms $\phi = \phi^1_H$ for $H \in \sm{M,\R}$ time independent. In this case $\psi_q = \phi^1_{\frac{1}{q} H}$ satisfies the conclusion of the proposition since for all $q$ sufficiently large $\gamma_{\F_q}(\frac{1}{q} H) = \frac{1}{q} (\max H - \min H),$ which clearly converges to zero. See Remark \ref{rmk: estimate} for an estimate of the convergence rate in Proposition \ref{prop: conv to zero}.

\begin{rmk}\label{rmk: manifolds}
The proofs of Theorem \ref{cor: SHS1} and Proposition \ref{prop: conv to zero} extend by rather straightforward, yet somewhat lengthy, technical modifications to the case where $(M,\om)$ is symplectically Calabi-Yau, that is, when $\langle c_1, A \rangle = 0$ for all $A \in \pi_2(M),$ to the case where $[\om]=0$ on $\pi_2(M),$ and to the case where $(M,\om)$ is negatively monotone, that is $[\om] = - \kappa c_1,$ $\kappa > 0,$ on $\pi_2(M)$ and, for technical reasons, semi-positive.  For ease of exposition, in this paper we limit ourselves to the symplectically aspherical case. Indeed, arguably the most interesting example, the standard symplectic $T^{2n},$ is already in this class of manifolds. However, we will present the arguments in these cases in a subsequent publication. 
\end{rmk}

\begin{rmk}\label{rmk: mcg}
By additional arguments slightly extending those of \cite{AS-flux} and \cite{SW-St} it is possible to extend Theorem \ref{cor: SHS2} to actions by homeomorphisms in $\ol{\Symp}(M,\om),$ the $C^0$-closure in $\mrm{Homeo}(M)$ of the full group $\Symp(M,\om)$ of symplectomorphisms. Since this would require a few technical detours from the subject of this paper, we will present these arguments in a subsequent publication.
\end{rmk}

Finally, we remark on the techniques used in the proof. Prior proofs of results pertaining to the Hilbert-Smith conjecture rely either on the direct geometric structure of orbits of the $\Z_p$ action or on Yang's theorem \cite{Yang} regarding the homological dimension of the orbit space. In contrast, our method uses the theory of barcodes of Hamiltonian maps. It was first used in symplectic topology in \cite{PS} and has since burgeoned to an active subfield of research. We will give only a very partial list of references. Of particular relevance to this paper are the references \cite{SZhao, S-HZ, SW-St, AS-nt}, which combine the theory of barcodes with equivariant power operations, with applications to Hamiltonian dynamics. These power operations can be considered to be Floer-theoretical manifestations of Smith theory \cite{Smith, Smith-2,Smith-3}. We also rely on \cite{BHS} proving that on closed symplectically aspherical manifolds, barcodes extend to Hamiltonian homeomorphisms. 

An important role is played by Hamiltonian spectral invariants, and spectral norm in particular. We briefly recall a few useful properties of these spectral invariants referring to the above references, as well as to the original definitions \cite{Viterbo, Schwarz, Oh} for more details and the proofs thereof, and to \cite{BHS} for the $C^0$ continuity property. Let $(M,\om)$ be a closed symplectically aspherical manifold. First of all, for a unital ring $\bK,$ a non-degenerate Hamiltonian diffeomorphism $\phi,$ namely whose graph $\gr(\phi) \subset M \times M$ intersects the diagonal transversely, and a choice of Hamiltonian $H \in \sm{S^1 \times M}$ generating it, as well as a generic $\om$-compatible almost complex structure $J$, Floer theory associates a graded complex $(CF^*(H;\bK),d)$ on the space $\oplus_{x \in \mrm{Fix}_c(\phi)} \bK(x)$ filtered by the action functional \[\cl A_H(x) = \int_0^1H(t,x(t)) dt -  \int_{\ol{x}} \om,\] where $x(t) = \phi^t_H(x),$ $\mrm{Fix}_c(\phi)$ are those fixed points of $\phi$ such that $x(t)$ is contractible, and $\ol{x}: \bb D \to M,$ $\bb D = \{ |z| \leq 1 \} \subset \C$ is a smooth map contracting $x(t).$ The filtration is defined as \[\cl A_H( \sum a_i x_i) = \min \{ \cl A_H(x_i)\;|\; a_i \neq 0\},\] and $\cl A_H(0) = +\infty.$ Then $\cl A_H(dx) > \cl A_H(x)$ and $CF^*(H;\bK)^{>a} = \cl A_H^{-1}( (a,\infty) )$ is a subcomplex with respect to $d.$ We write $HF^*(H;\bK)$ for the homology of $(CF^*(H;\bK),d)$ and $HF^*(H;\bK)^{>a}$ for the homology of $CF^*(H;\bK)^{>a}.$ By \cite{PSS} there is an isomorphism \[ PSS_H: H^*(M;\bK) \to HF^*(H;\bK) \] and there are maps $\pi^{a,\infty}: HF^*(H;\bK)^{>a} \to HF^*(H;\bK)$ for all $a \in \R.$ Now for a class $\al \in H^*(M;\bK) \setminus \{0\}$ set $\al_H = PSS_H(\al).$ Then \[ c_{\bK}(\al, H) = \sup \{a \in \R\;|\; \al_H \in \mrm{Im}(\pi^{a,\infty})\}.\] These invariants satisfy the following properties: 

\begin{enumerate}
\item (Hofer-continuity) For $F,G \in \sm{S^1 \times M, \R},$ \[-\int_0^1 \max(F_t-G_t)\,dt \leq c_{\bK}(\al, F) - c_{\bK}(\al, G) \leq \int_0^1 \min(F_t-G_t)\,dt.\]
\item (Super-additivity) For $F,G \in \sm{S^1 \times M, \R},$ $\al, \beta \in H^*(M;\bK),$ \[ c_{\bK}(\al \cup \beta, F \# G) \geq  c_{\bK}(\al, F) + c_{\bK}(\beta, G),\] where $F\#G(t,x) = F(t,x)+G(t,(\phi^t_F)^{-1}(x))$ is a Hamiltonian generating the flow $(\phi^t_F \phi^t_G)_{t \in [0,1]}.$
\item (Independence of the Hamiltonian) For $H \in \sm{S^1 \times M, \R}$ and the inverse Hamiltonian $\ol{H}(t,x) = - H(1-t, x),$ \[ \gamma_{\bK}(\phi) = - c_{\bK}(1, H) - c_{\bK}(1, \ol{H}) \geq 0\] depends only on the time-one map $\phi = \phi^1_H.$ 
\item (Non-degeneracy) For all $\phi \in \Ham(M,\om),$ $\phi \neq \id,$ $\gamma_{\bK}(\phi) > 0.$   
\item (Poincar\'e duality) Let $\bK$ be a field. Then for all $H \in \sm{S^1 \times M, \R},$ \[c_{\bK}(1, H) = - c_{\bK}(\mu, \ol{H}),\] where $\mu \in H^{2n}(M;\Z)$ is the positive generator and we keep the same notation for the induced class with all coefficient rings.
\item ($C^0$ continuity) $\gamma_{\bK}$ is uniformly $C^0$ continuous and extends to $\ol{\Ham}(M,\om).$ Moreover it remains non-degenerate on $\ol{\Ham}(M,\om).$
\end{enumerate}

Note that the Hofer-continuity and the Independence of the Hamiltonian properties imply that for all $\phi \in \Ham(M,\om),$ \begin{equation}\label{eq: ga hof}\gamma_{\bK}(\phi) \leq ||\phi||, \end{equation} where \[||\phi|| = \inf_{\phi^1_H = \phi}  \int_0^1 \left(\max(H_t)-\min(H_t)\right)\,dt \] is the Hofer norm of $\phi.$

\section*{Acknowledgments}
I thank Leonid Polterovich for numerous inspiring discussions related to the subject of this paper as well as for kindly offering to attach Appendix \ref{app: Leonid} to it. I thank him, Marcelo S. Atallah, and Nicholas Wilkins for fruitful collaborations which have greatly influenced this work. I also thank Sebastian Hurtado and Jarek K{\c{e}}dra for initially bringing the Hilbert-Smith conjecture to my attention in 2017. I was supported by an NSERC Discovery Grant, the Fondation Courtois, and a Sloan Research Fellowship.

\section{Proof of Theorem \ref{cor: SHS1}}

Let $\bK$ be a field. Let us recall a few properties of barcodes of Hamiltonian homeomorphisms from \cite{BHS} in cohomological form. For a Hamiltonian homeomorphism $\phi \in \Ham(M,\om)$ there exists a barcode $\cl B(\phi; \bK)_i = \{(I_{i,j}, m_{i,j})_{j \in \N}\},$ in degree $i \in \Z$ with coefficients in $\bK,$ defined up to shift, where $I_{i,j}$ are itervals of the form $[a,b)$ or $(-\infty, a)$ and $m_{i,j} \in \N$ are positive integers. There are exactly $\dim_{\bK} H^i(M;\bK)$ infinite bars of the form $(-\infty,a)$ counting with multiplicity. All other bars are finite of the form $[a,b).$ There is also a total barcode $\cl B(\phi; \bK)$ of $\phi$ also well-defined up to shift, so that in a natural way $\cl B(\phi; \bK) = \sqcup_{i \in \Z} \cl B(\phi; \bK)_i,$ the sum (in the sense of multi-sets) of the barcodes over the degrees (more precisely, there exist barcode representatives, not up to shift, of all of these barcodes, which satisfy this identity). Crucially for this paper, $\cl B(\phi; \bK)$ satisfies the following finiteness condition. Set \[\cl N_{\eps}(\phi;\bK) = \# \{ \text{bars in}\; \cl B(\phi; \bK)\; \text{of length} > \eps \}.\] Then \[ \cl N_{\eps}(\phi;\bK) < \infty\] for all $\eps>0.$ This finiteness condition also follows from inequality \eqref{eq: approx} below.

Let $\beta_i(\phi)$ denote the {\em boundary depth} of $\phi$ in degree $i.$ Namely it is the maximal length of a bar in $\cl B(\phi;\bK)_i.$ By the finiteness property it exists, is finite, and in view of \cite{KS} satisfies $\beta_i(\phi) \leq \gamma(\phi;\bK)$ for all $i.$ In particular \[\cl N'_{\eps}(\phi;\bK) = \# \{ i \in \Z\;|\; \beta_i(\phi;\bK) > \eps\}\] satisfies \[ \cl N'_{\eps}(\phi;\bK) \leq \cl N_{\eps}(\phi;\bK) < \infty.\]

\begin{lma}\label{lma: Lefschetz}
Let $\phi \in \ol{\Ham}(M,\om),$ $\eps > 0$ and $\cl Q$ an infinite set of primes. Then there exist $q_0(\phi), q_1(\phi,\eps) \in \N,$ such that for all $q \in \cl Q,$ $q \geq q_0(\phi),$ $\gamma_{\F_q}(\phi) \geq \gamma_{\Q}(\phi)/2$ and for all $q \in \cl Q,$ $q \geq q_1(\phi,\eps),$ $\cl N_{\eps}(\phi;\F_q) \leq \cl N_{\eps/2}(\phi; \Q).$
\end{lma}

\begin{proof}[Proof of Lemma \ref{lma: Lefschetz}]
Fix $\delta>0.$ We can assume that $\phi \neq \id$ and $\gamma_{\Q}(\phi)>0.$ By the $C^0$-continuity of the spectral norm, there exists a non-degenerate Hamiltonian {\em diffeomorphism} $\psi = \psi_{\delta}$, such that $|\gamma_{\bK}(\psi)-\gamma_{\bK}(\phi)| < \delta$ for every field $\bK.$ However, by the proof of \cite[Lemma 16]{S-HZ}, $\gamma_{\F_q}(\psi) = \gamma_{\Q}(\psi)$ for all $q \geq q'(\psi).$ Therefore for such $q,$ \[ \gamma_{\F_q}(\phi) > \gamma_{\F_q}(\psi) - \delta = \gamma_{\Q}(\psi) - \delta > \gamma_{\Q}(\phi) - 2\delta.\] Setting $\delta = \gamma_{\Q}(\phi)/4 > 0$ and $q_0(\phi) = q'(\psi_{\delta})$ then finishes the first part of the proof. To provide further details, the proof of \cite[Lemma 16]{S-HZ} shows that for all $q \geq q'(\psi),$ we have an identity of barcodes $\cl B(H;\F_q) = \cl B(H; \Q)$ for any Hamiltonian $H$ generating $\psi.$ In particular the maximal end-points of infinite bars agree, so $c_{\F_q}(\mu,H) =c_{\Q}(\mu,H).$  Moreover, the same is true for the minimal end-points of infinite bars, so $c_{\F_q}(1,H) =c_{\Q}(1,H).$ Hence, by the Poincar\'e duality property, \[c_{\F_q}(1,\ol{H}) = -c_{\F_q}(\mu,H) =-c_{\Q}(\mu,H) = c_{\Q}(1,\ol{H}).\] Therefore $\gamma_{\F_q}(\psi) = -c_{\F_q}(1,H) - c_{\F_q}(1,\ol{H})$ and $\gamma_{\Q}(\psi) = -c_{\Q}(1,H) - c_{\Q}(1,\ol{H})$ agree for all $q \geq q'(\psi).$

For the second part of the proof, fix $\eps/2 > \delta > 0.$ By the definition of the barcode of $\phi,$ there exists a non-degenerate Hamiltonian diffeomorphism $\psi=\psi_{\delta},$ such that \begin{equation}\label{eq: approx} \cl N_{\eps}(\phi;\bK) \leq \cl N_{\eps-\delta}(\psi;\bK) \leq \cl N_{\eps-2\delta}(\phi;\bK).\end{equation} (See for example \cite{BHS}, \cite[Section 2]{CC}.) Now, by \cite[Lemma 16]{S-HZ}, \[\cl N_{\eps-\delta}(\psi;\F_q) =  \cl N_{\eps-\delta}(\psi;\Q)\] for all $q \geq q''(\psi).$ Therefore for such $q,$ \[\cl N_{\eps}(\phi;\F_q) \leq \cl N_{\eps-\delta}(\psi;\F_q) =  \cl N_{\eps-\delta}(\psi;\Q) \leq \cl N_{\eps-2\delta}(\psi;\Q).\] Therefore setting $\delta = \eps/4$ and $q_1(\phi,\eps) = q''(\psi_{\delta})$ finishes the proof.



\end{proof}

\begin{lma}\label{lma: reduction}
Let $\phi \in \ol{\Ham}(M,\om).$ Then for every prime $q,$ $\gamma_{\F_q}(\phi) \leq \gamma_{\Z}(\phi).$
\end{lma}

\begin{proof}
By \cite{KS-Z}, we have $\gamma_{\F_q}(\psi) \leq \gamma_{\Z}(\psi)$ for all primes $q$ and Hamiltonian diffeomorphisms $\psi.$ Consider $\psi_j \in \Ham(M,\om)$ such that $d_{C^0}(\psi_j, \phi) \to 0$ as $j \to \infty.$ Then for every ring $\bK,$ $\gamma_{\bK}(\phi) = \displaystyle \lim_{j \to \infty} \gamma_{\bK}(\psi_j).$ Hence we obtain the desired inequality by passing to the limit as $j \to \infty.$ 

To provide further detail, let us prove the inequality for $\psi \in \Ham(M,\om).$ By the Hofer-continuity of spectral invariants, it is enough to prove it for non-degenerate $\psi.$ Let $H$ be a Hamiltonian generating $\psi.$ It is enough to prove that $c_{\F_q}(1,H) \geq c_{\Z}(1,H).$ Let $x = \sum a_i x_i,$ where $x_i$ are the (canonically) capped periodic orbits, be a cycle homologous to $y = PSS_{\Z, H}(1) = \sum b_i x_i$ in the Floer complex $CF^*(H;\Z),$ such that $\cl A_H(x) = c_{\Z}(1, H).$ Let $z \in CF^*(H;\Z),$ $z = \sum c_i x_i$ be such that $dz = x-y.$ Let $r_q: CF^*(H;\Z) \to CF^*(H;\F_q)$ be the reduction map modulo $q.$ Then $r_q(x) = \sum r_q(a_i) x_i$ is homologous to $r_q(y) = \sum r_q(b_i) x_i = PSS_{\F_q, H}(1)$ in $CF^*(H;\F_q)$ since $d r_q(z) = r_q(dz) = r_q(x) - r_q(y).$ Moreover $\cl A_H(r_q(x)) \geq \cl A_H(x).$ Therefore $c_{\F_q}(1,H) \geq \cl A_H(r_q(x)) \geq c_{\Z}(1,H)$ as desired.
\end{proof}

Analyzing the proof of the main inequality in \cite{SW-St}, we obtain the following sharpening thereof. Indeed, it is easy to see that the bound on the drop in filtration at the $i$-th step, $i = 1, \ldots, 2n(q-1)$ in the argument proving \cite[Theorem 1.15]{SW-St} is bounded by $\beta_{2nq+1 - i}(\phi^q;\F_q).$ The rest of the proof proceeds similarly.

\begin{thm}\label{thm: ineq}
Let $\psi \in \ol{\Ham}(M,\om)$ and let $q$ be a prime. Then \[ q \cdot \gamma(\psi; \F_q) - \gamma(\psi^q; \F_q) \leq 2 \sum_{2n < i \leq 2nq}  \beta_i(\psi^q; \F_q).\] 
\end{thm}

\begin{proof}[Proof of Theorem \ref{cor: SHS1}]
Let $\rho: \Z_p \to \ol{\Ham}(M,\om)$ be an effective continuous action with respect to the $C^0$ topology on the target. In particular it is continuous with respect to the metric $\gamma_{\Z}$ on the target. Let $\phi = \rho(1) \in \ol{\Ham}(M,\om).$ Set $\eps_{-} = \gamma_{\Q}(\phi)/4.$ Let $a_k:\Z_p \to \Z/p^k$ denote the standard projection coming from the isomorphism $\Z_p/p^k\Z_p \cong \Z/p^k.$ By the continuity of the action there exists $k \in \bb N$ such that for all $x \in p^k \Z_p,$ $\gamma_{\Z}(\rho(x)) < \eps_-.$ Consider the set $\cl Q$ of primes $q \equiv 1\mod p^k.$ It is an infinite set due to Dirichlet's theorem on primes in arithmetic progressions.

First note that as $\Z_p$ is compact and the action is continuous, $\eps_{+} = \max_{x \in \Z_p} \gamma_{\Z}(\rho(x))$ is attained and is finite. Moreover $\eps_{+,q} = \max_{x \in \Z_p} \gamma_{\F_q}(\rho(x))$ for a prime $q$ is again attained and satisfies $\eps_{+,q} \leq \eps_+$ by Lemma \ref{lma: reduction}.  Then, similarly, for $q \in \cl Q,$ $\eps_{-,q} = \min_{a_k(x) = 1} \gamma_{\F_q}(\rho(x))$ is attained and is strictly positive. Moreover, by the triangle inequality and Lemma \ref{lma: Lefschetz} $\eps_{-,q} \geq \gamma_{\F_q}(\phi) - \eps_{-} \geq \eps_{-}$ for all $q\geq q_0(\phi),$ a lower bound independent of $q.$ So in summary, for $q \geq q_0(\phi),$ \[ 0 <  \eps_- \leq \eps_{-,q} \leq \eps_{+,q} \leq \eps_{+} < \infty.\] Now observe that for $q \in \cl Q,$ its tautological image $q \in \Z \subset \Z_p$ satisfies $a_k(q) =1$ and is invertible in $\Z_p$. Let $q^{-1} \in \Z_p$ be its inverse. Then $a_k(q^{-1}) = 1$ and $\psi_q = \rho(q^{-1})$ satisfies $(\psi_q)^{q} = \rho(q^{-1} q) = \rho(1) = \phi,$ so $\psi_q$ is a $q$-th root of $\phi.$ Let us apply Theorem \ref{thm: ineq} for $\psi = \psi_q$ and $\psi^q = \phi.$ We get \[ q \cdot \gamma(\psi_q; \F_q) - \gamma(\phi; \F_q) \leq 2 \sum_{2n < i \leq 2nq}  \beta_i(\phi; \F_q).\] Now for $q \geq q_0(\phi),$ \[\gamma(\psi_q; \F_q) \geq \eps_-,\] and for all $q,$ \[\beta_i(\phi;\F_q) \leq \gamma(\phi;\F_q) \leq \eps_+.\] Now for $c>0$ to be chosen later, let $I_{c,q} = \{ 2n < i \leq 2nq\;|\; \beta_i(\phi;\F_q) \leq c \}$ and $n_{c,q} = \# I_{c,q}.$ Note that $n_{c,q} \leq 2n(q-1).$ Then for all $q \geq \max\{q_0(\phi), q_1(\phi,c)\},$ we have by Lemma \ref{lma: Lefschetz} that \[\cl N_{c}(\phi;\F_q) \leq \cl N_{c/2}(\phi;\Q).\] Therefore \[ q \eps_- - \eps_+ \leq 2 c n_{c,q} + 2\cl N_{c}(\phi;\F_p) \eps_+ \leq 4n(q-1)c + 2\cl N_{c/2}(\phi;\Q) \eps_+,\] as there are at most $\cl N_c(\phi;\F_q)$ indices $i$ (in particular among $2n<i \leq 2nq$) with $\beta_i(\phi;\F_q)>c.$ Rearranging the terms, we get \[ q (\eps_- - 4nc) \leq (2 \cl N_{c/2}(\phi;\Q)+1) \eps_+.\] Fixing a number $c < \frac{\eps_-}{4n},$ for instance $c=\frac{\eps_-}{8n},$ and considering $q \in \cl Q,$ $q \geq \max\{q_0(\phi), q_1(\phi,c)\},$ we get that the left hand side grows linearly in $q,$  while the right hand side is bounded in $q.$ This is a contradiction, which completes the proof.
\end{proof}

\begin{rmk}\label{rmk: estimate}
The corresponding estimate for Proposition \ref{prop: conv to zero} is \[\gamma_{\F_q}(\psi_q) \leq \frac{1}{q} (2\cl N_{\gamma_{\F_q}(\psi_q)/8n}(\phi;\Q)+1) \gamma_{\Z}(\phi),\] for large primes $q,$ where $\psi_q \in \ol{\Ham}(M,\om)$ satisfies $(\psi_q)^q = \phi.$ For instance if the barcode counting function satisfies $\cl N_{\delta}(\phi;\Q) \leq C(\phi) \delta^{-\al},$ $\al > 0,$ for a constant $C(\phi)$ depending on $\phi$ then \[\gamma_{\F_q}(\psi_q) \leq C(\phi, \al, n) q^{-1/(1+\al)}\] for a constant $C(\phi, \al, n)$ depending only on $\phi, \al, n$ and $q$ large. According to \cite{CGG-bar}, if $\phi \in \Ham(M,\om)$ is smooth, then this estimate holds with $\al = 4n.$ An extension of Remark \ref{rmk: new} shows that for $\phi \in \ol{\Ham}(M,\om),$ $\cl N_{\delta}(\phi;\Q)$ can grow arbitrarily fast as $\delta \to 0.$
\end{rmk}

\begin{rmk}\label{rmk: completion}
The argument proving Theorem \ref{cor: SHS1} and Proposition \ref{prop: conv to zero} works verbatim for $\ol{\Ham}(M,\om)$ with the $C^0$ metric replaced by the image $\hat{\Ham}^{\Z,\Q}(M,\om)$ of the natural map $\hat{\Ham}^{\Z}(M,\om) \to \hat{\Ham}^{\Q}(M,\om)$ from the $\gamma_{\Z}$-completion of $\Ham(M,\om)$ to its $\gamma_{\Q}$-completion. (This map exists since $\gamma_{\Z} \geq \gamma_{\Q}$ on $\Ham(M,\om)$ by \cite{KS-Z}.) We consider this group equipped with the corresponding extension of $\gamma_{\Z}.$ It would be interesting to see if these arguments could be extended to the completion $\hat{\Ham}(M,\om)$ of the Hamiltonian group with respect to Hofer's metric.
\end{rmk}

\section{Outlook}

Besides the directions already described in Remarks \ref{rmk: manifolds} and \ref{rmk: mcg}, it would be interesting to extend Theorem \ref{cor: SHS1} to positively monotone or in fact general symplectic manifolds. Note that in this case one has to rule out faithful actions, since $\ol{\Ham}(M,\om)$ has finite subgroups even for $M=S^2.$ It would also be interesting to reduce new cases of the Hilbert-Smith conjecture for non-symplectic homeomorphisms to results already proven in the symplectic setting. Finally, in view of Proposition \ref{prop: conv to zero} and the general approach of \cite{PS}, it would be interesting to see if a Hamiltonian diffeomorphism, or homeomorphism, admitting roots of all orders is necessarily autonomous.
\appendix

\section{On actions of compact semisimple $p$-adic Lie groups}\label{app: Leonid}
\centerline{by Leonid Polterovich$^{1}$}
\footnotetext[1]{ Partially supported by the Israel Science Foundation grant 1102/20.}

\bigskip

\subsection{Invariant Hilbert-Smith groups}
Call a separable topological group $D$ {\em invariant} if its topology is induced by a bi-invariant metric.
We say that a topological group $D$ is {\it Hilbert-Smith} if it does not admit a continuous
monomorphism from the group of $p$-adic integers $\Z_p$.

\begin{thm}\label{thm-2}  Let $G$ be a compact semisimple $p$-adic analytic Lie group. Then every homomorphism $f: G \to D$, where $D$ is an invariant Hilbert-Smith group has finite image. In particular, if $\dim G \geq 1,$ then $f$ cannot be a monomorphism.
\end{thm}

The key point is that the monomorphism is not assumed to be continuous. 
This result follows directly from \cite{PSST} and the main objective of this text is to help the reader to navigate through this paper.

\begin{exa}{\rm The group $\Ham(M,\om)$ of Hamiltonian diffeomorphisms of a closed symplectically aspherical manifold $(M,\om)$ with the Hofer metric is invariant Hilbert-Smith. Indeed, the separability of $\Ham(M,\om)$ with the Hofer metric is proved in \cite[Proof of Theorem 1.9]{KLM}. The non-existence of continuous monomorphisms of $\Z_p$ into $\Ham(M,\om)$ is the statement of Theorem \ref{thm: Hofer}, which could be interpreted as a Hofer-geometric symplectic Hilbert-Smith conjecture. When a compact semisimple $p$-adic analytic Lie group has no torsion, the statement of Theorem \ref{thm-2} seems to be new even in this {\it smooth} context. The existence of torsion in $\Ham(M,\om)$ is ruled out in \cite{P-nt}. Thus, in fact, Theorem \ref{thm-2} rules out {\em all} homomorphisms from compact semisimple $p$-adic analytic Lie group to this group. Another more technical example is the subject of Remark \ref{rmk: completion}.}
\end{exa}



\begin{rmk}\label{rem-nonc}  If $G$ is any (not necessarily compact) semisimple $p$-adic analytic Lie group, the kernel of any homomorphism $f: G \to D$ is open. This can be seen by restricting $f$ to an open compact subgroup existing by \cite[ Theorem 8.1]{Dixon}, and applying Theorem \ref{thm-2}. In particular, $f$ factors through a discrete group.
\end{rmk}

\begin{exa}\label{exam-main}  An example of a compact simple $p$-adic analytic Lie group is given by $SO(3,\Q_p)$, the group of orthogonal transformations with respect to a specially chosen definite quadratic form of $3$ variables on $\Q_p$. This group is studied in detail in \cite{DMPSW}.
It coincides with $SO(3,\Z_p)$, and hence is compact. By \cite[Corollary 4.3, Theorem 4.5]{Dixon}, it contains a torsion-free subgroup $G$ which is itself a compact simple $p$-adic analytic Lie group. We conclude that $G$ cannot act by Hamiltonian diffeomorphisms on closed aspherical symplectic manifolds. It is instructive to compare this to the obstructions appearing in the earlier literature. The paper \cite{P-nt} rules out actions of groups containing torsion and, roughly speaking,  unipotent elements (see \cite[1.6.H,1.6.L]{P-nt}). The latter do not exist in $G$ since every element of  $SO(3,\Q_p)$ is diagonalizable over the algebraic closure of $\Q_p$. Furthermore, Theorem 1.9 in \cite{KLM} rules out actions on {\it all} closed symplectic manifolds of semisimple algebraic group over an uncountable algebraically closed field. It is unknown to us whether $SO(3,\Q_p)$ can effectively act on a closed symplectic manifold,
as the field $\Q_p$ is not algebraically closed. 
\end{exa}

For other obstructions to abstract group actions by Hamiltonian diffeomorphisms, see eg \cite{P-nt, KLM}.

\subsection*{Acknowledgment} I thank Egor Shelukhin for all his help with this appendix, Yehuda Shalom for explaining to me Remark \ref{rem-nonc}, and Jarek Kedra for a useful comment leading to Example \ref{exam-main}.

\subsection{The glossary} We collect a couple of definitions from the theory of $p$-adic analytic Lie groups. A $p$-adic analytic Lie group $G$ is
a $p$-adic analytic manifold equipped with $p$-adic analytic group operations.
The Lie algebra $\mathfrak{g}$ is the tangent space to $G$ at the identity equipped with the Lie bracket. Let us mention that  the exponential map $\exp$
is not defined on the whole $\mathfrak{g}$, but it is an analytic injection of a neighborhood of $0$ in $\mathfrak{g}$ to $G$  (see \cite[Corollary 18.19]{Sch}).

We say that $G$ is {\it (semi)simple} if its Lie algebra is (semi)simple, i.e., the direct sum of simple algebras, where a Lie algebra is simple if it is non-abelian and has no non-trivial ideals. The terminological caveat is that a simple $p$-adic analytic Lie group is not a simple group: it has non-trivial normal subgroups.

\subsection{Proof of Theorem \ref{thm-2}}
First assume that $G$ is simple. By \cite[Theorem 3.3]{PSST}, $G$ is norm complete. Therefore, by the automatic continuity property (see \cite[Proposition 1.13]{PSST} and \cite[Proposition 1.14]{PSST}) $f$ must be continuous. Set $K = \text{Ker}(f)$. Then $K$ is a closed normal subgroup of $G$ and hence $G/K$ is a compact $p$-adic analytic Lie group by \cite[Theorem 9.6]{Dixon}. Moreover $f$ descends to a continuous embedding $G/K \to D.$ By a basic property of the exponential map we see that if $G/K$ has dimension $\geq 1$, then it contains an embedded copy of $\Z_p$. Indeed by \cite[Corollary 18.19]{Sch} the exponential map is a homeomorphism from a small open neighborhood of $0$ to its image and this neighborhood contains a copy of $\Z_p.$ However, this contradicts the Hilbert-Smith property of $D$. Thus, $G/K$ is $0$-dimensional, and hence finite. Hence $f$ has finite image.
If $G$ is semisimple, by \cite[Proof of Theorem 3.3]{PSST} it has a finite-index subgroup $G' = \prod_{i=1}^N G'_i$ which is a direct product of finitely many compact simple $p$-adic analytic groups $G'_i.$ Hence $f(G')$ is finite, since $f(G'_i)$ are finite, and hence $f(G)$ is finite. This finishes the proof of the first statement. 
Furthermore if $\dim G \geq 1,$ then $G$ is infinite while $f$ has finite image. Hence $f$ cannot be a monomorphism.
\qed

\bibliographystyle{plainurl}
\bibliography{bibliographySHS}
\noindent\\

\end{document}